\documentclass[a4paper,10pt]{amsart}
\usepackage[left=2.7cm,right=2.7cm,top=3.5cm,bottom=3cm]{geometry}

\usepackage[T1]{fontenc}
\usepackage[latin1]{inputenc}
\usepackage{lmodern}
\usepackage[english]{babel}
\usepackage[autostyle]{csquotes}

\usepackage{amsthm,amssymb,amsmath,amsfonts,mathrsfs,amscd,amsbsy,dsfont,verbatim}
\usepackage[all,cmtip]{xy}
\usepackage{mathtools}
\usepackage{graphicx}
\usepackage{mathrsfs}
\usepackage{centernot}
\usepackage{enumerate}
\usepackage{latexsym}
\usepackage{mathbbol}
\usepackage{stmaryrd}
\usepackage{tikz}
\usepackage{tikz-cd}
\usepackage{pdflscape}
\usepackage[mathcal]{euscript}
\usepackage[hidelinks]{hyperref}

    \DeclareFontFamily{U}{wncy}{}
    \DeclareFontShape{U}{wncy}{m}{n}{<->wncyr10}{}
    \DeclareSymbolFont{mcy}{U}{wncy}{m}{n}
    \DeclareMathSymbol{\Sha}{\mathord}{mcy}{"58} 

\theoremstyle{plain}
\newtheorem{theorem}{Theorem}[section]
\newtheorem{corollary}[theorem]{Corollary}
\newtheorem{lemma}[theorem]{Lemma}

\newtheorem{proposition}[theorem]{Proposition}

\newtheorem{innertheoremintro}{Theorem}
\newenvironment{theoremintro}[1]
  {\renewcommand\theinnertheoremintro{#1}\innertheoremintro}
  {\endinnertheoremintro}

\theoremstyle{definition}
\newtheorem{definition}[theorem]{Definition}
\newtheorem{assumption}[theorem]{Assumption}

\theoremstyle{remark}
\newtheorem{remark}[theorem]{Remark}
\newtheorem{example}[theorem]{Example}

\def\Q{\mathbf Q}
\newcommand{\Ql}{\Q_{\ell}}
\def\Z{\mathbf Z}
\newcommand{\Zl}{\Z_{\ell}}

\def\Alg{\ul{\mr{Al}}\mr{g}}

\def\GL{\mathrm{GL}}

\def\ul#1{\underline{#1}}

\def\mr#1{\mathrm{#1}}
\def\im{\mr{im}}

\def\Xcal{\mathcal{X}}
\def\Lcal{\mathcal{L}}
\def\Rcal{\mathcal{R}}
\def\Ccal{\mathcal{C}}
\def\Dcal{\mathcal{D}}
\def\Bbf{\mathbf{B}}
\def\Rep{\underline{\mathrm{Re}}\mathrm{p}}

\def\vect{\underline{\mathrm{vect}}}
\def\AbVar{\underline{\mathrm{AbVar}}}
\newcommand{\res}{\mathrm{res}}
\newcommand{\bD}{D_L}
\renewcommand{\hom}{\mr{Hom}}

\DeclareMathOperator{\Aut}{Aut}
\DeclareMathOperator{\Sel}{Sel}
\DeclareMathOperator{\Gal}{Gal}

\DeclareMathOperator{\Ram}{Ram}
\DeclareMathOperator{\End}{End}
\DeclareMathOperator{\Ind}{Ind}

\DeclareMathOperator{\Twist}{Twist}

\definecolor{Indigo}{rgb}{0.2,0.1,0.7}
\definecolor{Violet}{rgb}{0.5,0.1,0.7}
\definecolor{White}{rgb}{1,1,1}
\definecolor{Green}{rgb}{0.1,0.9,0.2}

\title[Local-global principle for twists]{On the local-global principle for twists of abelian varieties and Galois representations}

\author{Nirvana Coppola}
\address{Department of Mathematics, Universit\`{a} di Padova}
\email{nirvana.coppola@unipd.it}

\author{Lorenzo La Porta}
\address{Department of Mathematics, Universit\`{a} di Padova}
\email{lorenzo.laporta@protonmail.com}

\author{Matteo Longo}
\address{Department of Mathematics, Universit\`{a} di Padova}
\email{matteo.longo@unipd.it}

\begin{document}


\begin{abstract}
This paper investigates the validity of a local-global principle for finite twists of a large class of objects endowed with a continuous action of the absolute Galois group of a given number field, such as abelian varieties, modular forms and Galois representations. Our aim is to determine when, for $m$ a positive integer, twists that are given locally by characters of order $m$ are realised by a global character of the same order. We define and study a ``Tate--Shafarevich cohomology set'' that governs the obstruction to the local-global principle for $m$-atic twists and we prove that this set is finite. Finally, we apply our results and prove several instances of the local-global principle in various examples.
\end{abstract}

\maketitle

\tableofcontents

\section{Introduction}

Given a twist of an ``object'' defined over a number field (e.g.\ an abelian variety up to isogeny) satisfying certain local constraints, it is natural to ask whether these constraints are satisfied globally. In this paper, we investigate this problem in the case when the twists are given by some cyclic group of automorphisms of the object. We do so with a view towards two main concrete examples: abelian varieties and Galois representations. Let us first discuss in more detail the case of abelian varieties.

Let \(K\) be either a \(p\)-adic or a number field and write \(G_K \coloneqq\Gal(\overline{K}/K)\) for the absolute Galois group of \(K\), where \(\overline{K}\) is a fixed algebraic closure of \(K\). 
Let \(A\) and \(B\) be two abelian varieties defined over \(K\). We say that \(B\) is a \emph{twist} (up to isogeny) of \(A\) if there exists a finite extension \(M\) of \(K\) such that \(A_M\) and \(B_M\) are isogenous, where \(A_M\), respectively \(B_M\), denotes the base change of \(A\), respectively \(B\), to \(M\). Write \(D_{\overline K} \coloneqq \End(A_{\overline{K}}) \otimes_\Z \Q\), where 
\(\End(A_{\overline{K}})\) is the endomorphism ring of \(A_{\overline{K}}\). The twists of \(A\) are in bijection with elements of the cohomology set \(H^1(G_K, D_{\overline{K}}^\times)\). 
Let \(m \geq 1\) be an integer, denote by \(\mu_m\) the group of \(m\)-th roots of unity in \(\overline{\Q}\) and let \(\zeta_m\in\mu_m\) be a primitive \(m\)-th root of unity. Assume that \(D_{\overline{K}}\) is a \(\Q(\zeta_m)\)-algebra and that the subgroup \(\mu_m \subseteq D_{\overline{K}}^\times\) is \(G_K\)-stable. 
We say that \(B\) is an \emph{\(m\)-atic twist} of \(A\), if \(B\) is a twist that corresponds to an element of \(\im(H^1(G_K, \mu_m) \to H^1(G_K, D_{\overline{K}}^\times))\).
\par Assume now that \(K\) is a number field and write \(K_v\) for the completion of \(K\) at any place \(v\). We say that \(B\), a twist of \(A\), is \emph{locally \(m\)-atic}, if \(B_{K_v}\) is an \(m\)-atic twist of \(A_{K_v}\), for all finite places \(v\) of \(K\). We say that \((A,m)\) satisfies the \emph{local-global principle} if the following implication holds for any abelian variety \(B\) over \(K\): 
if \(B\) is a locally \(m\)-atic twist of \(A\), then \(B\) is an \(m\)-atic twist of \(A\). 
\par By Faltings's theorem, \cite{faltings}, the local-global principle holds for \(m = 1\) and all \(A\). The case \(m = 2\) was studied in \cite{ACF, fitebbasta}. In that setting, the local-global principle holds for abelian varieties of dimension at most \(3\). On the other hand, there are counterexamples in dimension \(4\) and, assuming that \(A\) is geometrically simple, in dimension \(12\).
In this work, we mainly focus on positive results in the case when \(m \geq 3\) is odd.

In analogy with the theory of rational points of abelian varieties, to study local-global properties, we introduce the two cohomological objects \(\Sel_m(G_K, D_{\overline{K}}^\times)\) and \(\Sha_m(G_K, D_{\overline{K}}^\times)\), which we call \emph{Selmer} and \emph{Tate--Shafarevich} sets, respectively (see Definition \ref{def : selsha}). These sets are designed to control the failure of the local-global principle for \(m\)-atic twists. We also remark that when \(D_{\overline{K}}^\times\) is commutative, these are groups, but, in general, they are just pointed sets. 
The failure of the local-global principle for \((A,m)\) is equivalent to 
\(\Sha_m(G_K, D_{\overline{K}}^\times)\) having more than one point. We investigate the set 
\(\Sha_m(G_K, D_{\overline{K}}^\times)\) by 
combining cohomological techniques and results from the theory of abelian varieties, and our first result is the finiteness of this obstruction. 
\begin{theoremintro}{A}[{Theorem \ref{theorem: finite sha}}]
 The pointed set \(\Sha_m(G_K, D_{\overline{K}}^\times)\) is finite. 
\end{theoremintro}
This result essentially follows from the representation theory of the skew group ring \(D_{\overline{K}}^\times \rtimes G\), where \(G\) is the finite quotient through which \(G_K\) acts on \(D_{\overline{K}}^\times\).

We now explain some positive results on the local-global principle for \((A,m)\). 
Denote by \(\End(A)\) the endomorphism ring of \(A\), and set \(D_K \coloneqq \End(A) \otimes_\Z \Q\). We mainly work with \(A\) geometrically simple, in one of the two extremal cases: either \(D_K^\times \cap \mu_m\) is trivial, Section \ref{sec: simp CM}, or \(\mu_m \subset D_K^\times\), see Section \ref{ssec: muminvar}, \ref{ssec: further apps}. We mention here two significant examples of what we prove.
\begin{theoremintro}{B}[Theorem \ref{thm: main_Gal}]
Let \(A\) be geometrically simple, and assume that \(D_{\overline{K}}\) is commutative. For all odd \(m \ge 3\) such that \(D_K^\times \cap \mu_m\) is trivial, \((A,m)\) satisfies the local-global principle.    
\end{theoremintro}
\begin{theoremintro}{C}[Corollary \ref{cor : easylgp}]
Let \(A\) be geometrically simple of dimension \(g\), such that \(\mu_m \subseteq D_K^\times\). For all \(g \leq 8\) and odd \(m \ge 3\), \((A,m)\) satisfies the local-global principle. 
\end{theoremintro}

Theorem C is mainly a consequence of a more general result, Theorem \ref{thm : coprimality}, in which we show that, under the condition \(\mu_m\subseteq D_K^\times\) (as in Theorem C), the local-global principle for \((A,m)\) holds when \(m\) is prime to the integer \(d = \frac{2\dim(A)}{[Z:\Q]}\), where \(Z\) is the center of \(D_K\). 

Even if so far we have only discussed our results in the case of abelian varieties, 
our definitions (e.g.\ Tate--Shafarevich and Selmer sets) and techniques apply \emph{verbatim} also to Galois representations.  
Let \(F\) be a field and \(V\) a finite-dimensional \(F\)-vector space. We say that two Galois representations \(\rho_1, \rho_2 \colon G_K \to \GL_F(V)\) are \emph{potentially equivalent}, or \emph{(finite) twists} of each other, if they are equivalent when restricted to a finite index subgroup. In that case, we say that the twist is \emph{\(m\)-atic} if there is a character \(\chi \colon G_K \to F^\times\) of finite order dividing \(m\) such that \(\rho_1 \sim \chi \rho_2\). Moreover, when \(K\) is a number field, \(\rho_1\) is a \emph{locally \(m\)-atic twist} of \(\rho_2\) if, for any decomposition group \(G_{K_v} \subseteq G_K\), \({\rho_1}_{|G_{K_v}}\) is an \(m\)-atic twist of \({\rho_2}_{|G_{K_v}}\). For a representation \(\rho\colon G_K \to \GL_F(V)\) and \(m\ge 1\) an integer, we say that \((\rho, m)\) satisfies the \emph{local-global principle} if any locally \(m\)-atic twist of \(\rho\) is \(m\)-atic. In the case of Galois representations attached to modular forms, we prove the following. 
\begin{theoremintro}{D}[{Corollary \ref{corollary: lgp for mod forms}}]
Let \(f \in S_{K}(N, \epsilon),\) be a normalised cuspidal newform of weight \(k \ge 2\), level \(N\), nebentypus \(\epsilon\) and Hecke field \(F_{f}\). For \(\lambda\) a prime of \(F_{f_i}\) dividing \(\ell\), take \(\rho_f \colon G_\Q \to \GL_2(F_{f, \lambda})\) the associated Galois representation. Then, \((\rho_f, m)\) satisfies the local-global principle for all \(m \ge 1\).
\end{theoremintro}

Both abelian varieties and Galois representations are instances of a more general framework that we formalise using the language of pseudo-functors and descent, see Section \ref{sec: galois descent}. This framework, which makes the notion of ``object'' mentioned at the beginning more precise, allows us to compare the local-global principle for abelian varieties and their associated Galois representations, see Lemma \ref{lemma: if lgp for rep then lgp for abvar}. 

\subsection*{Acknowledgements}
The authors are grateful to Emiliano Ambrosi, Julian Demeio, Francesc Fit\'e, Andrea Gallese, Davide Lombardo, and Matteo Verzobio for fruitful conversations about this project.
\par NC and LLP are supported by the \emph{``Piano di Sviluppo Dipartimentale''} with the title \emph{``Enhancing the Research in the Math Dept'', CUP C93C24000160005}. ML is supported by PRIN 2022. NC and ML are members of the Indam group GNSAGA.

\section{Setup and notation}\label{sec: setup}

Let \(p,\, \ell\) be primes, fixed throughout the paper.

\subsubsection*{Category theory}
Let \(\Ccal\) be a category and \(X, Y,\) objects of \(\Ccal\). We use the usual notations \(\hom_{\Ccal}(X, Y)\), \(\mr{Isom}_{\Ccal}(X, Y)\), \(\End_{\Ccal}(X)\), \(\Aut_{\Ccal}(X),\) for the various obvious sets of morphisms. We write \(\ul{\mr{Cat}}\) for the category of all small\footnote{Here \emph{small} means an element of some fixed Grothendieck universe.} categories with functors as morphisms. For any group \(H\), we write \(\Bbf H\) for the associated groupoid  with one object (or \emph{classifying space} of \(H\)), that is, the category with a unique object \(\ast\) and \(\End_{\Bbf H}(\ast) = H\). 

\subsubsection*{Fields of definition}
In what follows, \(K\) will denote a perfect field. In \S\ref{sec: twist of a rep}, from Assumption \ref{ass: number field} onwards, \(K\) will be a number field. We will also consider the case when \(K\) is a \(p\)-adic field, that is, a finite extension of \(\Q_p\). Fix \(\overline{K}\) an algebraic closure of \(K\). 
For any extension \(M/K\) contained in \(\overline{K}\), we write \(G_M = \Gal(\overline{K}/M)\), which is naturally a subgroup of \(G_K = \Gal(\overline{K}/K)\).

\subsubsection*{Number fields} Suppose that \(K\) is number field. We denote by \(\Sigma_K\) the set of places of \(K\). Write \(\Sigma^\infty_K \subset \Sigma_K\) for the set of places \(v \nmid \infty\), that is, the finite places of \(K\). For each \(v \in \Sigma_K\), we denote by \(K_v\) the completion of \(K\) at \(v\) and by \(k(v)\) the residue field of \(K_v\), of which we fix an algebraic closure \(\overline{k}(v)\). We fix, for all \(v \in \Sigma_v\), an algebraic closure \(\overline{K}_v\) of \(K_v\) and an embedding \(\iota_v \colon \overline{K} \to \overline{K}_v\). The embedding \(\iota_v\) induces an injection \(G_{K_v} \to G_K\), for all \(v \in \Sigma_K\), so we can identify \(G_{K_v}\) with a subgroup of \(G_K\),  which we call a \emph{decomposition group} of \(G_K\) at \(v\).

If \(A\) is a group, not necessarily abelian, with a continuous action of \(G_K\) with respect to the discrete topology on \(A\), we write \(H^1(G, A)\) for the pointed set of continuous Galois cohomology. See \cite[\S5]{Serre-Galois} for more details.


\subsubsection*{Fields of coefficients and representations} Let \(F\) be a field, which we regard as the ``field of coefficients'', and \(V\) a vector space over \(F\) of finite dimension \(n\). We fix an algebraic closure \(\overline{F}\) of \(F\). 

We consider representations \(\rho\colon G_K \to \GL_F(V)\). We write \(\End_F(V)\) for the \(F\)-linear endomorphisms of \(V\) and \(\End_F(\rho)\) for the 
\(F\)-linear endomorphisms of \(V\) as \(G_K\)-module via \(\rho\). Let \(G_K\) act on \(\End_F(V)\) via \(\sigma \mapsto (f\mapsto \rho(\sigma)f\rho(\sigma)^{-1})\). Therefore, \(\End_F(V)^{G_K}=\End_F(\rho)\) and, more generally, \(\End_F(V)^{G_M}=\End_F(\rho_{|G_M})\),
for any extension \(M/K\) contained in \(\overline{K}\). 
We will  often write \(D_M(\rho)\), or \(D_M\), instead of \(\End_F(\rho_{|G_M})\) and \(D\) instead of \(D_K\).

If \(F\) is an \(\ell\)-adic field (e.g.\ \(E_\lambda\), for \(E\) a number field and \(\lambda \mid \ell\), like in \S\ref{sssec: extra end}), we will assume, unless otherwise stated, that \(\rho\) is continuous with respect to the Krull topology on \(G_K\) and the \(\ell\)-adic topology on \(\GL_F(V)\). In that case, we say that \(\rho\) is \emph{geometric} if it is semisimple, unramified outside a finite set of places \(\Ram(\rho) \subset \Sigma_K\), and de Rham at all \(v \in \Sigma_K\) such that \(v \mid \ell\), see \cite{fontainereppadic, fontainemazur}. In practice, we can ignore the condition at the places lying over \(\ell\).

\subsubsection*{Roots of unity} For a field \(E\) and an integer \(m \ge 1\), denote \(\mu_m(E)\) the group of \(m\)-th roots of unity in \(E\). We write \(\mu_m = \mu_m(\overline{F})\). Let \(\zeta_m \in \mu_m\) denote a primitive \(m\)-th root of unity, fixed throughout the paper.

\subsubsection*{Abelian varieties}
Let \(A\) be an abelian variety over \(K\). For a field extension \(M/K\), we denote by \(\End(A_M)\) the ring of endomorphisms of \(A\) defined over \(M\). We set \(\End^0(A_{M}) \coloneqq \End(A_{M}) \otimes_\Z \Q\). Similarly, if \(B\) is another abelian variety, we write \(\hom^0(A_M, B_M) \coloneqq \hom(A_M, B_M) \otimes_\Z \Q\) for the morphisms of abelian varieties up to isogeny. If \(M=K\), we simply write \(\hom^0(A, B)\) and \(\End^0(A)\).
We have an action of \(G_K\) on \(\End(A_{\overline{K}})\), defined as follows. 
For a point \(P\in A(\overline{K})\) and an isogeny \(\psi : A\rightarrow A\), 
we set \((\sigma(\psi))(P) \coloneqq \sigma(\psi(\sigma^{-1}(P)))\). By scalar extension, we have an induced action of \(G_K\) on \(\End^0(A_{\overline{K}})\) and, by definition, \(\End^0(A_{\overline{K}})^{G_K}=\End^0(A_{K})\). We denote the \(\ell\)-adic Tate module of \(A\) by \(V_\ell(A) \coloneqq T_\ell(A) \otimes_{\Zl} \Ql\). This is also endowed with a natural action of \(G_K\).
As in the case of representations, we will  often write \(D_M(A)\), or \(D_M\), instead of 
\(\End^0(A_M)\) and \(D\) instead of \(D_K\).

\subsubsection*{Pointed sets} For \(f \colon (X, x) \to (Y, y)\) a map of pointed sets, we write \(\ker(f) \coloneqq f^{-1}(y)\).

\section{Galois descent}\label{sec: galois descent}
For this section, let \(K\) be perfect and \(F\) any field. Our aim is to set up the minimal amount of formalism needed to talk about ``objects with a Galois action''. We achieve this using the language of pseudo-functors and descent data, focusing on the case of Galois descent. This level of generality will be useful to compare different classes of objects with a Galois action, most importantly abelian varieties and their Galois representations, and to then study their twists. A standard reference is \cite{FGA}.

\begin{definition}
Consider \(\Ccal\) a category. A \emph{categorical action} of a group \(H\) on the category \(\Ccal\) is the datum of a pseudo-functor \(T \colon \Bbf H \to \ul{\mr{Cat}}\) that sends \(\ast\) to \(\Ccal\).
\end{definition}
This a special case of a \emph{fibred category}, \cite[Definition 1.1]{FGA}. The definition means that for all \(g \in H\), we have \(T(g) \colon \Ccal \to \Ccal\) an autoequivalence and, for \(g, h \in H\), we have a natural isomorphism
\[
    \alpha_{g,h} \colon T(g) \circ T(h) \xrightarrow{\sim} T(gh),
\]
as well as \(\mr{Id}_T \colon T(1_H) \xrightarrow{\sim} \mr{id}_\Ccal\), 
satisfying natural compatibilities.
In what follows, we always assume that \(\Ccal\) is an \(F\)-linear abelian category and that the \(T(g)\)'s are exact \(F\)-linear functors. 
\begin{definition}
    A \emph{descent datum} for \(X \in \Ccal\) with respect to \(T\) is a family \(\varphi\) of isomorphisms \(\varphi_g \colon T(g)(X) \to X\) such that, for all \(g, h \in H\), the diagram
\[
\begin{tikzcd}
    & T(g)(T(h)(X)) \arrow[d, "{\alpha_{g, h}}(X)"] \arrow[rr, "\varphi_g \circ T(g)({\varphi_{h}})"] && X \\
    & T(gh)(X) \arrow[urr, "{\varphi_{gh}}", swap]
\end{tikzcd}
\]
commutes and \(\varphi_{\mr{id}_H} = \mr{Id}_T(X)\). We denote the datum of an object \(X \in \Ccal\) with a fixed descent datum \(\varphi\) by \((X, \varphi)\).
\end{definition}
Given \((X, \varphi),  (Y, \psi)\) objects of \(\Ccal\) with descent data, we can define an action of \(H\) on \(\hom_\Ccal(X, Y)\) by taking, for all \(g \in H\) and \(f \in \hom_\Ccal(X, Y)\), the composition
\[
    g(f) \coloneqq \psi_g \circ T(g)(f) \circ \varphi_g^{-1}.
\]
One can check that this defines a group action. It follows from our assumptions that this action is always \(F\)-linear. To ease notations, when discussing this action on \(\hom_{\Ccal}(X, Y)\), we omit the descent data \(\varphi, \psi\), even though the action itself depends on them. Notice that one can think of an element \(f \in \hom_{\Ccal}(X, Y)^{H}\) as a morphism of descent data, because, for all \(g \in H\), we have the commutative diagram
\[\begin{tikzcd}
	{T(g)(X)} & X \\
	{T(g)(Y)} & Y.
	\arrow["{\varphi_g}", from=1-1, to=1-2]
	\arrow["{T(g)(f)}"', from=1-1, to=2-1]
	\arrow["f", from=1-2, to=2-2]
	\arrow["{\psi_g}"', from=2-1, to=2-2]
\end{tikzcd}\]
We say that \((X, \varphi)\) is \emph{equivalent} to \((Y, \psi)\) \emph{via \(f\)} if there is an isomorphism \(f \in \hom_{\Ccal}(X, Y)^{H}\). More generally, if \(H' \le H\) is a subgroup, we say that \((X, \varphi)\) is \emph{equivalent} to \((Y, \psi)\) \emph{over \(H'\)} if there is an isomorphism \(f \in \hom_{\Ccal}(X, Y)^{H'}\).

We are mostly interested in the case \(H = G_K\), so, from now on, we restrict to that. For \(G_M \subseteq G_K\), where \(M\subseteq \overline{K}\) is a finite extension of \(K\), we will talk of morphisms, or equivalences, of descent data ``over \(M\)'', instead of over \(G_M\).
\begin{remark}
One should think of an object \(X\) of \(\Ccal\) as an object ``defined over \(\overline{K}\)'' and a pair \((X, \varphi)\), where \(\varphi\) is a descent datum, as an object ``defined over \(K\)'' underlying \(X\).
\end{remark}
The group \(G_K\) carries a natural topology which we want to take into consideration.
For \((X, \varphi),  (Y, \psi),\) as above, the action of \(G_K\) on \(\hom_\Ccal(X, Y)\) is continuous with respect to the discrete topology on the latter, if and only if the natural inclusion
\[
    \bigcup_{M/K} \hom_\Ccal(X, Y)^{G_M} \subseteq \hom_\Ccal (X, Y),
\]
is an equality, where \(M \subseteq \overline{K}\) runs over the finite extensions of \(K\).
Moreover, if \(M\subseteq M'\) is an inclusion of finite extensions of \(K\), then \(\End_{\Ccal}(X)^{G_M} \subseteq \End_{\Ccal}(X)^{G_{M'}}\). A routine check shows that \(\bigcup_{M/K} \End_\Ccal(X)^{G_M}\) is always an \(F\)-linear sub-algebra of \(\End_\Ccal(X)\).

We can thus view 
\[
    \End_{\Ccal}(X)^\mr{sm} \coloneqq \bigcup_{M/K} \End_\Ccal(X)^{G_M}
\]
as the sub-algebra of \emph{smooth vectors} for the action given by \(T\) and the descent datum \(\varphi\). One can, more generally, for \((X, \varphi),  (Y, \psi),\) define
\[
    \hom_{\Ccal}(X, Y)^\mr{sm} \coloneqq \bigcup_{M/K} \hom_\Ccal(X, Y)^{G_M}.
\]
The algebra \(\End_{\Ccal}(X)^\mr{sm}\) inherits a continuous action of \(G_K\).

For \(M\) a finite extension of \(K\), we will sometimes write, like in the case of Galois representations, \(D_M(X, \varphi)\), or \(D_M\) if \((X, \varphi)\) is clear from context, instead of \(\End_{\Ccal}(X)^{G_M}\).

\section{Minimal endomorphism field}\label{section: min endo field}

Our aim is to define the notion of \emph{minimal endomorphism field}. For our purposes, the main cases of interest are those of \(F\)-linear representations of \(G_K\) and abelian varieties over \(K\) up to isogeny. Nevertheless, we set things up in greater generality, to make the connection between the two, a priori distinct, notions as plain as possible.

\subsection{General definition}
Let \(T\) and \((X, \varphi)\) be as in \S\ref{sec: galois descent}.

\begin{definition}\label{def: min endo field}
 We define the \emph{minimal endomorphism field} of \((X, \varphi)\) as the Galois extension \(L_{\mr{end}}(X, \varphi)\) of \(K\) fixed by the kernel of the action of \(G_K\) on \(\End_\Ccal(X)^\mr{sm}\).
\end{definition}
Again, to ease notations, we sometimes write simply \(L_{\mr{end}}(X)\) and omit the descent datum \(\varphi\).

Let \(\Dcal\) be another \(F\)-linear abelian category with a categorical action \(S \colon \Bbf G_K \to \ul{\mr{Cat}}\), \(S(\ast) = \Dcal\), by exact \(F\)-linear auto-equivalences on \(\Dcal\). There is a natural way to relate the categorical actions \(T\) and \(S\). 
\begin{definition}
A \emph{realisation of \(T\) in \(S\)} is the datum of an \(F\)-linear exact functor \(R \colon \Ccal \to \Dcal\) endowed with a natural isomorphism \(\eta(\sigma, X) \colon R(T(\sigma)(X)) \cong S(\sigma)(R(X))\), functorial in \(X \in \Ccal\), for all \(\sigma \in G_K\), such that we have natural compatibilities
\begin{align*}
    \eta(\sigma\tau, X) \circ R(\alpha_{\sigma, \tau}^T(X)) &= \alpha_{\sigma, \tau}^S(X) \circ S(\sigma)(\eta(\tau, X)) \circ \eta(\sigma, T(\tau)(X)),\\
    R(\mr{Id}_T(X)) &= \mr{Id}_S(R(X)) \circ \eta(\mr{id}_{G_K}, X),
\end{align*}
functorial in \(X \in \Ccal\), for all \(\sigma, \tau \in G_K\).
\end{definition}

We collect here some formal results.
\begin{lemma}\label{lemma: min endo gen prop} Let \(T\) and \((X, \varphi)\) be as above.
\begin{enumerate}
    \item \label{lemma: min endo gen prop 1} If \(\End_{\Ccal}(X)^\mr{sm}\) has finite dimension over \(F\), then \(L_{\mr{end}}(X)\) is finite over \(K\).
    \item \label{lemma: min endo gen prop 2} If \(E\) is an extension of \(F\), then \(T\) induces a categorical action \(T \otimes_F E\) of \(G_K\) on \(\Ccal \otimes_F E\) and \(L_{\mr{end}}(X) = L_{\mr{end}}(X \otimes_F E)\).
    \item \label{lemma: min endo gen prop 3} Let \(\Dcal\) and \(S\) be as above and suppose that \(R \colon \Ccal \to \Dcal\) is a fully faithful realisation of \(T\) in \(S\). Then \(L_{\mr{end}}(X, \varphi) = L_{\mr{end}}(R(X), \{R(\varphi_\sigma) \circ \eta(\sigma, X)^{-1}\}_{\sigma \in G_K})\).
\end{enumerate}
\end{lemma}
\begin{proof}
\begin{enumerate}
    \item Pick an \(F\)-basis \(f_1, f_2, \ldots, f_m \in \End_{\Ccal}(X)^\mr{sm}\). For each \(i, 1 \le i \le m\), we have a finite extension \(M_i\) of \(K\) such that \(f_i \in \End_{\Ccal}(X)^{G_{M_i}}\). Moreover, if \(M_i'\) is another finite extension such that \(f_i \in \End_{\Ccal}(X)^{G_{M_i'}}\), then \(f_i \in \End_{\Ccal}(X)^{G_{M_i \cap M_i'}}\), by the fundamental theorem of Galois theory. In particular, we can take \(M_i\) to be minimal for each \(i\). Then, \(L_{\mr{end}}(X) \subseteq \overline{K}\) is the smallest finite field extension of \(K\) that contains all of the \(M_i\)'s, which is finite.
    \item Notice that \(\End_{\Ccal}(X \otimes_F E)^{G_M} = (\End_{\Ccal}(X)\otimes_F E)^{G_M} = \End_{\Ccal}(X)^{G_M} \otimes_F E\). Moreover, since tensor products commute with colimits, we have 
    \begin{align*}
        \End_{\Ccal}(X)^\mr{sm} \otimes_F E &= \left(\bigcup_{M/K} \End_\Ccal(X)^{G_M}\right)\otimes_F E \\
        &= \bigcup_{M/K} (\End_\Ccal(X)^{G_M}\otimes_F E) = \End_{\Ccal}(X \otimes_F E)^\mr{sm}.
    \end{align*}
    Also, for \(M, M'\) finite extensions of \(K\) in \(\overline{K}\), \(\End_{\Ccal}(X)^{G_M} \subseteq \End_{\Ccal}(X)^{G_M'}\) if and only if \(\End_{\Ccal}(X\otimes_F E)^{G_M} \subseteq \End_{\Ccal}(X\otimes_F E)^{G_M'}\). This is enough to conclude.
    \item The functor \(R\), for each \(X \in \Ccal\), gives an isomorphism \(\End_{\Ccal}(X) \xrightarrow{\sim} \End_{\Dcal}(R(X)),\) which, by the definition of realisation, is compatible with the action of \(G_K\), given on the left by the descent datum \(\varphi\) and on the right by the descent datum \(R(\varphi_\sigma) \circ \eta(\sigma, X)^{-1}\) on \(R(X)\). In particular, the \(G_M\)-invariants are in functorial (in \(X\)) bijection and the minimal endomorphism fields clearly coincide. \qedhere
\end{enumerate}
\end{proof}

\subsection{Representations}\label{ssec: cat action on reps}
Consider \(\rho \colon G_K \to \GL_F(V)\) a representation as in \S\ref{sec: setup}.
The functor \(S \colon \Bbf G_K \to \ul{\mr{Cat}}\) that sends \(\ast\) to \(\vect_F\), the category of finite-dimensional \(F\)-vector spaces, and \(\sigma \in G_K\) to \(S(\sigma)= \mr{id}_{\vect_F}\) defines a categorical action of \(G_K\) on \(\vect_F\) (where we take \(\alpha_{\sigma, \tau}^S = \mr{id}_{\vect_F} = \mr{Id}_S\)). To \(\rho\) we can associate a descent datum on \(V \in \vect_F\), simply by taking \(\varphi_\sigma = \rho(\sigma)\). The same holds for any object of \(\Rep_F(G_K)\), the category of \(F\)-linear representations of \(G_K\) into finite-dimensional \(F\)-vector spaces. In fact, a descent datum for \(S\) is the same as an element of \(\Rep_F(G_K)\) and a morphism (respectively, an equivalence) of objects with descent data is just a morphism (respectively, an isomorphism) of representations.

Recall that for any finite extension \(M\) of \(K\), we write \(D_M(\rho) = D_M = \End_{F}(\rho_{|G_M})\). We have \(\End_F(V)^\mr{sm} = \bigcup_{M/K} D_M \subseteq \End_F(V)\).
\begin{lemma}\label{lemma: exist min endo field}
There is a well-defined unique minimal extension \(L\) of \(K\) such that \(\End_F(\rho)^\mr{sm} = D_L\). Moreover, \(L\) is finite and Galois over \(K\).
\end{lemma}

\begin{proof}
This is essentially Lemma \ref{lemma: min endo gen prop}.(\ref{lemma: min endo gen prop 1}), together with the fact that \(\End_F(V)^\mr{sm} \subseteq \End_F(V)\) is finite-dimensional.
\end{proof}

The field from the previous lemma is precisely the minimal endomorphism field of the representation \(\rho\).
A priori, if \(E\) is an extension of \(F\), \(L_{\mr{end}}(\rho\otimes_F E)\) might differ from \(L_{\mr{end}}(\rho)\), but, by Lemma \ref{lemma: min endo gen prop}.(\ref{lemma: min endo gen prop 2}), they coincide.

\begin{definition}[{\cite[Section 1.4]{fitebbasta}}]
We say that \(\rho\) is \emph{strongly absolutely irreducible} if \(\rho_{|G_M} \otimes_F E\) is irreducible for all \(M\) finite extensions of \(K\) and \(E\) finite separable extension of \(F\).
\end{definition}

\begin{lemma}\label{lemma: strongly abs irred min end field}
If \(\rho\) is strongly absolutely irreducible, then \(D_L = D_K=F\).
\end{lemma}
\begin{proof}
This follows from Schur's lemma: since \(\rho_{|G_M}\) is absolutely irreducible, 
\(D_M=F\) for all finite extensions \(M\) of \(K\). 
\end{proof}

\subsection{Abelian varieties}\label{ssec: min endo for abvars}
For this subsection, suppose that \(K\) is a number field. Let \(A\) be an abelian variety over \(K\).

We can consider \(\AbVar_{\overline{K}}^0\) the category of abelian varieties over \(\overline{K}\) with morphisms given by morphisms of abelian varieties up to isogeny. Then \(\AbVar_{\overline{K}}^0\) is endowed with a natural categorical action \(T\) of \(G_K\) where \(T(\sigma)(X) = (\sigma^{-1})^\ast(X)\), for \(X\) an abelian variety over \(\overline{K}\). In this setting, by \cite{sga1, weildescent}, any descent datum is \emph{effective} in the sense that giving a descent datum on \(X\) is the same as giving an abelian variety \(X_0\) over \(K\) with a fixed isomorphism \(X_{0, \overline{K}} \cong X\). Thus, we may identify \(A\) with a descent datum on \(A_{\overline{K}}\). Then, the minimal endomorphism field, as defined above, coincides with the \emph{endomorphism field} from \cite{guralnickkedlaya}. This is a Galois extension \(L/K\), finite by Lemma \ref{lemma: min endo gen prop}.(\ref{lemma: min endo gen prop 1}), such that \(\End(A_{\overline{K}})=\End(A_{L})\) and \(\End(A_M)=\End(A_{L})\) for all finite field extensions \(M/L\). If we want to stress the dependence on $A$, we write $L_\mr{end}(A)$ for $L$. Like in Definition \ref{def: min endo field}, we see that \(G_{L}\) is the kernel of the action of \(G_K\) on \(\End^0(A_{\overline{K}})\).

By Faltings's theorem \cite{faltings}, if \(B\) is another abelian variety over \(K\),
\[
    \hom^0(A, B)\otimes_\Q \Ql \cong \hom_{\Ql[G_K]}(V_\ell(A), V_\ell(B)).
\]
In particular, for all \(A\), we see that the \(\ell\)-adic representation \(V_\ell(A)\) of \(G_K\) is semisimple and \(\End^0(A_{M}) \otimes_{\Q} \Ql \cong \End_{\Ql[G_M]}(V_\ell(A))\), for any \(M\) finite extension of \(K\).

\subsubsection{Extra endomorphisms}\label{sssec: extra end}
Take \(E\) any commutative subfield of \(\End^0(A_{K})\) central in \(\End^0(A_{\overline{K}})\). The action of \(G_K\) on \(\End^0(A_{\overline{K}})\) is \(E\)-linear and the minimal endomorphism field \(L_{\mr{end}}(A)\) does not change if we consider \(\End^0(A_{M})\) as an \(E\)-algebra for all \(M/K\) finite.
Moreover, even if \(E\) is strictly larger than \(\Q\), the isomorphism \(\End^0(A_{M}) \otimes_{\Q} \Ql \cong \End_{\Ql[G_M]}(V_\ell(A))\) is one of \(E \otimes_\Q \Ql\)-algebras, for all \(M\) as above. In fact, if we write \(E \otimes_\Q \Ql \cong \prod_{\lambda \mid \ell} E_\lambda\), where \(E_\lambda\) is the completion of \(E\) at the prime \(\lambda\), then we also have isomorphisms \(\End^0(A_{M}) \otimes_{E} E_{\lambda} \cong \End_{E_{\lambda}[G_M]}(V_\lambda(A))\), for the usual meaning of \(V_\lambda(A)\).

\subsubsection{Minimal fields}
Let \(E\) be a number field. Consider \(\AbVar_{\overline{K}}^E\) the category of abelian varieties over \(\overline{K}\) with the extra condition that \(E\) is a commutative subfield central in \(\End^0(A_{\overline{K}})\) and morphisms given by morphisms up to isogeny that commute with the \(E\)-algebra structures. Notice that the definition of descent datum applied to this category, which is \(E\)-linear, implies that \(E\) is fixed by the action of \(G_K\). Let \(A\) be abelian variety over \(K\) such that \(E\) is a commutative subfield of \(\End^0(A_{K})\) central in \(\End^0(A_{\overline{K}})\). Up to isomorphism, this is the same as giving an object of \(\AbVar_{\overline{K}}^E\) with a descent datum. Let \(\lambda \mid \ell\) be a prime of \(E\) and \(\rho_{A, \lambda}\) the usual representation on \(V_\lambda(A)\).
\begin{lemma}\label{lemma: min endo of ab is repth}
We have \(L_{\mr{end}}(A) = L_{\mr{end}}(\rho_{A, \lambda})\).
\end{lemma}
\begin{proof}
This follows from Faltings's theorem and Lemma \ref{lemma: min endo gen prop}.(\ref{lemma: min endo gen prop 2}), (\ref{lemma: min endo gen prop 3}), because \(V_\lambda\) is a fully faithful realisation from \(T \otimes_E E_\lambda\), the categorical action of \(G_K\) on \(\AbVar_{\overline{K}}^E\), to \(S\), the categorical action on representations. 
\end{proof}

\section{Twists}
Let \(\Ccal\) be an \(F\)-linear abelian category with a categorical action \(T\) of \(G_K\). We consider \((X, \varphi)\) an object of \(\Ccal\) with a fixed descent datum \(\varphi\). We write \(D_M\) instead of \(\End_{\Ccal}(X)^{G_M}\), for \(M\) a finite extension of \(K\), and \(L\) for the minimal endomorphism field of \((X, \varphi)\). In particular, \(\bD = \End_{\Ccal}(X)^\mr{sm}\).
\begin{definition}\label{def: general twist}
Let \((Y, \varphi')\) and \((Z, \varphi'')\) be objects of \(\Ccal\) with descent data.
\begin{enumerate}
\item We say that \((Y, \varphi')\) is a \emph{twist} of \((X, \varphi)\) if there is some \(M\) finite extension of \(K\) such that \((X, \varphi)\) is equivalent to \((Y, \varphi')\) over \(M\), as descent data. If this equivalence is given by an isomorphism \(f \in \mr{Isom}_{\Ccal}(X, Y)^\mr{sm}\), then we say that \((Y, \varphi')\) is a twist of \((X, \varphi)\) via \(f\).

\item If \((Z, \varphi'')\) is another twist of \((X, \varphi)\), we say that \((Y, \varphi')\) and \((Z, \varphi'')\) are \emph{equivalent} as twists of \((X, \varphi)\) if \((Y, \varphi')\) and \((Z, \varphi'')\) are equivalent as descent data over \(K\).

\item When \(K\) is a number field and \(v\) a place of \(K\), we say that \((Y, \varphi')\) is a \emph{\(G_{K_v}\)-twist}, or a \emph{local twist at \(v\)}, of \((X, \varphi)\) if it is a twist with respect to the categorical action of \(G_{K_v}\), obtained from \(T\) by restriction along \(\Bbf G_{K_v} \to \Bbf G_K\).

\item We write \(\mathcal{E}(X, \varphi)\) for the set of equivalence classes of twists of \((X, \varphi)\).
\end{enumerate}
\end{definition}

\begin{proposition}\label{proposition: twists and cohomology}
There is a natural bijection between \(\mathcal{E}(X, \varphi)\) and \(H^1(G_K, \bD^\times)\).
\end{proposition}
\begin{proof}
First, we show that there is a well-defined map \(\theta \colon \mathcal{E}(X, \varphi) \to H^1(G_K, \bD^\times).\)
Let \((Y, \varphi')\) be a twist of \((X, \varphi)\) over \(M\), via the isomorphism \(f \colon X \to Y \in \hom_{\Ccal}(X, Y)^{G_M}\). Consider the map
\begin{align*}
    c_f \colon G_K &\longrightarrow \bD^\times, \\
    \sigma &\longmapsto f^{-1} \sigma(f) = f^{-1}\varphi'_{\sigma}T(\sigma)(f)\varphi_{\sigma}^{-1}.
\end{align*}
One can check that \(c_f\) is a continuous one-cocycle. Indeed, \(c_f\) is a one-cocycle, because, for \(\sigma, \tau \in G_K\),
\[
    c_f(\sigma\tau) = f^{-1}((\sigma \tau)(f)) = f^{-1}(\sigma(ff^{-1} \tau(f))) = c_f(\sigma)\sigma(c_f(\tau)).
\]
To see that \(c_f(\sigma) \in \bD^\times \subseteq \Aut_{\Ccal}(X)\), consider \(\tau \in N\), with \(N\) the normal core of \(G_M\) (which is open and normal in \(G_K\)). Then
\[
    \tau(f^{-1} (\sigma(f))) = f^{-1} (\sigma \sigma^{-1} \tau \sigma)(f) = f^{-1} \sigma(f).
\]
If \(\tau \in N\), then \(c_f(\tau) = \mr{id}_X\). This shows that \(c_f\) is continuous. 

Suppose that \((Y, \varphi')\) is also a twist of \((X, \varphi)\) over \(M'\) via \(f'\in \hom_{\Ccal}(X, Y)^{G_{M'}}\). Then
\[
    (f')^{-1}fc_f(\sigma) = (f')^{-1} \sigma(f) = (f')^{-1} \sigma (f'(f')^{-1}f) = c_{f'}(\sigma) \sigma((f')^{-1}f).
\]
This shows that the class \([c_f] = [c_{f'}] \in H^1(G_K, \bD^\times)\) does not depend on \(f\).

Suppose that \((Z, \varphi'')\) is another twist of \((X, \varphi)\) and that \((Y, \varphi')\) is equivalent, as twist, to \((Z, \varphi'')\) via \(g \in \mr{Isom}_{\Ccal}(Y, Z)^{G_K}\). Then \(c_{gf}(\sigma) = c_f(\sigma)\). In particular, the class \([c_f]\) only depends on the equivalence class of the twist \((Y, \varphi')\). Thus, we can define \(\theta\) by sending the class of \((Y, \varphi')\) to the class \([c_f] \in H^1(G_K, \bD^\times)\).

We now construct an inverse \(\xi\) of \(\theta\). Suppose that \(\chi = [c] \in H^1(G_K, \bD^\times)\). Then one can check that \((c\varphi)_{\sigma} \coloneqq c(\sigma) \varphi_{\sigma}\) defines a descent datum on \(X\). By continuity of \(c\), there is an open normal subgroup \(N\) of \(G_K\) such that, for all \(\tau \in N\), \(c(\tau) = \mr{id}_X.\)
One can check that \((X, c \varphi)\) is equivalent to \((X, \varphi)\) via \(\mr{id}_X\) over \(\overline{K}^N\), that is, \((X, c \varphi)\) is a twist of \((X, \varphi)\).
If \(c'\) is co-homologous to \(c\), then there is some \(b \in \bD^\times\) such that \(c'(\sigma) = b^{-1} c(\sigma) \sigma(b)\), for all \(\sigma \in G_K\), thus
\(b c'(\sigma) \varphi_\sigma = c(\sigma) \varphi_\sigma T(\sigma)(b)\), which implies that \(c'\varphi\) and \(c\varphi\) are equivalent as twists of \((X, \varphi)\). Hence, the equivalence class of \((X, c\varphi)\) in \(\mathcal{E}(X, \varphi)\) only depends on \(\chi\). This defines a map \(\xi \colon H^1(G_K, \bD^\times) \to \mathcal{E}(X, \varphi).\)
Notice that, if \((Y, \varphi')\) is a twist of \((X, \varphi)\) via \(f\), then \((X, c_f\varphi)\) is a twist of \((X, \varphi)\) via \(\mr{id}_X\), which is equivalent to \((Y, \varphi')\), as twists, via \(f\). This means that \(\xi\theta\) is the identity.
Moreover, \(c(\sigma)\varphi_{\sigma}T(\sigma)(\mr{id}_X)\varphi_{\sigma}^{-1} = c(\sigma)\) shows that \(\theta\xi\) is also the identity.
\end{proof}

Let \(m \ge 1\) be an integer. From now on, we assume that \(F\) has characteristic zero and that the following holds.
\begin{assumption}\label{ass1}
We assume that \(\bD\) is finite-dimensional over \(F\), that it admits a structure of \(F(\zeta_m)\)-algebra and we fix one such structure. Moreover, we assume that \(\mu_m \subseteq \bD^\times\) is stable under the action of \(G_K\).
\end{assumption}

Thanks to Assumption \ref{ass1}, we can set
\[
    \Twist_m(G_K, \bD^\times) \coloneqq \im(H^1(G_K, \mu_m) \to H^1(G_K, \bD^\times)).
\]

\begin{assumption}\label{ass: number field}
From now on, we assume that \(K\) is a number field.
\end{assumption}

\begin{definition}\label{def-local-global-I}
Let \((X, \varphi), (Y, \varphi')\) be two objects with descent data.
\begin{enumerate}
\item We say that \((X, \varphi)\) and \((Y, \varphi')\) are \emph{(globally) \(m\)-atic twists} if there exists \(\chi \in \Twist_m(G_K, \bD^\times)\) such that \((Y, \varphi')\) is a twist of \((X, \varphi)\) in the class associated with \(\chi\) by Proposition \ref{proposition: twists and cohomology}. 
\item We say that \((X, \varphi)\) and \((Y, \varphi')\) are \emph{locally \(m\)-atic twists at \(v\)} if \((X, \varphi)\) is a twist of \((Y, \varphi')\) and there exists \(\chi_v \in \Twist_{m}(G_{K_v}, \bD^\times)\) such that the class of \((Y, \varphi')\) as a \(G_{K_v}\)-twist of \((X, \varphi)\) is associated with \(\chi_v\). 
\item We say that \((X, \varphi)\) and \((Y, \varphi')\) are \emph{locally \(m\)-atic twists} if \((Y, \varphi')\) is a twist of \((X, \varphi)\) and they are locally \(m\)-atic twists at \(v\) for all finite places \(v\) of \(K\).
\item We say that \((X, \varphi, m)\) satisfies the \emph{local-global principle} if all locally \(m\)-atic twists are \(m\)-atic twists.
\end{enumerate}
\end{definition}

\subsection{Galois representations}\label{sec: twist of a rep}
Consider a representation \(\rho: G_K \to \GL_F(V)\), with \(F\) any field, \(V\) as above. 
\begin{definition}
We say that \(\rho' \colon G_K \to \GL_F(V)\) is an \emph{(algebraic) twist} of \(\rho\), or that \(\rho\) and \(\rho'\) are \emph{potentially equivalent}, if there exists \(M\), finite extension of \(K\), such that \(\rho'_{|G_M} \sim \rho_{|G_M}\). Two algebraic twists \(\rho'\), \(\rho''\) of \(\rho\) are said to be \emph{equivalent} if \(\rho' \sim \rho''\) over \(G_K\).
\end{definition}

With the setup of \S\ref{ssec: cat action on reps}, one can see that these notions of twist and equivalence are compatible with the general ones given in Definition \ref{def: general twist}. In particular, instead of \(\End_{\vect_F}(V)^\mr{sm}\) (with respect to descent datum entailed by \(\rho\)), we will simply write \(\bD\), or \(\bD(\rho)\). We write that \((\rho, m)\) satisfies the local-global principle if the corresponding object with descent datum \((V, \varphi, m)\) does.
If \(\chi = [c] \in H^1(G_K, D_L^\times)\), we denote by \(\rho^\chi(\sigma) \coloneqq c(\sigma)\rho(\sigma)\), for \(\sigma \in G_K,\) the twist of \(\rho\) by \(\chi\).

\subsection{Abelian varieties}\label{sec: twist of an abvar}
Let \(A, B\) be abelian varieties over \(K\).
\begin{definition}
We say that \(B\) is a \emph{twist} of \(A\) if \(A_{\overline{K}}\) is isogenous to \(B_{\overline{K}}\).
\end{definition}
This notion of twist recovers that of Definition \ref{def: general twist} with the setup of \S\ref{ssec: min endo for abvars}. We write that \((A, m)\) satisfies the local-global principle if the corresponding object with descent datum \((A_{\overline{K}}, \varphi, m)\) does.
If \(\chi\in H^1(G_K, D_L^\times)\), then we denote by \(A^\chi\) the twist of \(A\) by \(\chi\).

\section{Tate--Shafarevich sets} \label{sec: def sha}
Let \((X, \varphi)\) be an object with a fixed descent datum in an \(F\)-linear abelian category \(\Ccal\) with a categorical action \(T\) satisfying Assumption \ref{ass1}. Let \(L\) denote the minimal endomorphism field of \((X, \varphi)\) and write \(D_L = \End_{\Ccal}(X)^{\mr{sm}} = \End_{\Ccal}(X)^{G_L}\). 
We have the exact sequence
\begin{equation}\label{tautological}
1 \longrightarrow \mu_m\longrightarrow \bD^\times\longrightarrow \bD^\times/\mu_m\longrightarrow 1\end{equation} 
Taking Galois cohomology, and using that \(\mu_m\) is normal in \(\bD^\times\), we obtain a commutative diagram of pointed sets with exact rows (\cite[Chapitre VII, Annexe, Proposition 2]{serre-locaux}, \cite[\S5]{Serre-Galois}): 
{\footnotesize {\[
\xymatrix{
1 \ar[r] & \mu_m^{G_K}\ar[r]\ar[d] &D^\times \ar[r]\ar[d] & \left(\bD^\times/\mu_m\right)^{G_K} \ar[r]^-{\nu}\ar[d] & \\
1\ar[r] & \prod_v\mu_m^{G_{K_v}}\ar[r] &\prod_v \bD^{\times,G_{K_v}} \ar[r] & \prod_v \left(\bD^\times/\mu_m\right)^{G_{K_v}}\ar[r] & 
}
\]}}
{\footnotesize {
\begin{equation}\label{dia1}
\xymatrix{
\ar[r]^-{\nu}& H^1(G_K,\mu_m) \ar[r]\ar[d]^-{\lambda}\ar@{-->}[rd]^{\vartheta} & H^1(G_K,\bD^\times) \ar[r]\ar[d]^-{\mu}\ar@{-->}[dr]^-{\psi} & 
H^1(G_K, \bD^\times/\mu_m) \ar[d]^-{\varepsilon}\ar[r]^-\delta&H^2(G_K,\mu_m)\ar[d]\\
\ar[r]& \prod_v H^1(G_{K_v},\mu_m) \ar[r] & \prod_v H^1(G_{K_v},\bD^\times) \ar[r] & \prod_v H^1(G_{K_v}, \bD^\times/\mu_m)\ar[r]&\prod_vH^2(G_{K_v},\mu_m)}
\end{equation}}}
where the vertical maps are products of the canonical restriction maps. 
\begin{definition}\label{def : selsha}
\
\begin{enumerate}
    \item The \emph{\(m\)-Selmer (pointed) set} of \(\rho\) is \(\Sel_m(G_K, \bD^\times)=\ker( \psi)\). 
     \item The \emph{\(m\)-Tate--Shafarevich (pointed) set} of \(\rho\) is \(\Sha_m(G_K, \bD^\times)=\ker( \varepsilon)\cap\ker(\delta)\). 
\end{enumerate} 
\end{definition}
It follows immediately from the definitions that there is an exact sequence 
\begin{equation}\label{FES}
 1\longrightarrow \Twist_m(G_K, \bD^\times)\longrightarrow
\Sel_m(G_K, \bD^\times)\longrightarrow \Sha_m(G_K, \bD^\times)\longrightarrow 1.\end{equation}

In particular, we have the following. 

\begin{lemma}
The triple \((X, \varphi, m)\) satisfies the local-global principle if and only if \(\Sha_m(G_K, \bD^\times)\) is reduced to a point.
\end{lemma}

\section{Results on the local-global principle}

Let \((X, \varphi)\) be as in \S\ref{sec: def sha}. 

\subsection{Reduction to finite Galois extensions}
Write \(G=\Gal(L/K)\) and, for \(v\) a finite place of \(K\), \(D_v \le G_K\) for the image of \(G_{K_v}\) inside \(G\), that is, the decomposition group of \(G\) at \(v\). Notice that, by Assumption \ref{ass1}, \(\bD\) has finite dimension over \(F\), so that \(L\) is finite over \(K\), by Lemma \ref{lemma: min endo gen prop}.(\ref{lemma: min endo gen prop 1}). In particular, \(G\) is finite.

Consider the commutative diagram (\cite[Chapter I, \S5.8, a)]{Serre-Galois}):
\begin{equation}\label{dia2}
\xymatrix{
1\ar[r] & H^1(G,\bD^\times/\mu_m)\ar[r]^-{\mathrm{inf}}\ar[d]^-{\phi} & H^1(G_K,\bD^\times/\mu_m)\ar[d]^-{\varepsilon}\ar[r]^-{\res}&H^1(G_L,\bD^\times/\mu_m)^G\ar[d]\\
1\ar[r] &\prod_vH^1(D_v,\bD^\times/\mu_m)\ar[r]^-{\mathrm{inf}} & \prod_vH^1(G_{K_v},\bD^\times/\mu_m)\ar[r]^-{\res}& 
\prod_vH^1(G_{L_v},\bD^\times/\mu_m)^{D_v}
}\end{equation}
in which we write \(\mathrm{inf}\) and \(\res\) for the various inflation and restriction maps in cohomology. 

\begin{lemma}\label{lemma}
   The inflation map identifies \(\ker(\phi)\) with \(\ker(\varepsilon)\). 
\end{lemma}
\begin{proof}
Clearly \(\mathrm{inf}(\ker(\phi))\subseteq\ker(\varepsilon)\). For the opposite inclusion, 
   using that \(G_L\) acts trivially on \(\bD\), we see that if \(x\in \ker(\varepsilon)\) then \(\res_{L/K}(x)\) belongs to the kernel of the product of the restriction maps 
    \[\hom(G_L,\bD^\times/\mu_m)/(\bD^\times/\mu_m)\longrightarrow \prod_{v\in \Sigma_L^\infty}\hom(G_{L_v},\bD^\times/\mu_m)/(\bD^\times/\mu_m).\]
   This kernel is trivial by Chebotarev's density theorem, so \(x\) is in the image of the inflation map. 
\end{proof}

\begin{proposition}\label{propnirvana}
Suppose that there exists a prime \(v\) in \(K\) such that \(D_v=G\). Then \((X, \varphi, m)\) satisfies the local-global principle for all integers \(m\).
\end{proposition}
\begin{proof} 
    Take \(x\in \Sha_m(G_K, \bD^\times)\). Then, by Lemma \ref{lemma}, \(x=\mathrm{inf}_{L/K}(y)\) for a unique element 
    \(y\) in \(H^1(G,\bD^\times/\mu_m)\) and the image
    of \(y\) 
    in \(H^1(D_v,\bD^\times/\mu_m)\) is trivial for all \(v\). 
    For \(v\) as in the statement, the map \(H^1(G,\bD^\times/\mu_m)\rightarrow H^1(D_v,\bD^\times/\mu_m)\)
    is the identity. Hence, \(y=1\), so \(x=1\).
\end{proof}
\begin{corollary}
\label{cor : inertiffcyc}
If \(G\) is cyclic, then the local-global principle holds for \((X, \varphi, m)\) for all \(m\). In particular, this happens if \(L = K\).
\end{corollary}
\begin{proof}
This follows by Proposition \ref{propnirvana} and Chebotarev's density theorem: if \(G\) is cyclic, there exist infinitely many unramified primes \(v \in \Sigma_K\) such that \(G = D_v\) (the inert primes).
\end{proof}

Taking the \(G\)-cohomology, respectively \(G_K\)-cohomology, of \eqref{tautological}, we obtain the commutative diagram with exact rows of pointed sets
\begin{equation}\label{diaxi}
\xymatrix
{
H^1(G,\mu_m)\ar[r]\ar[d]^-{\mathrm{inf}_{G/G_K}}& H^1(G,\bD^\times)\ar[r]\ar[d]^-{\mathrm{inf}_{G/G_K}} & H^1(G,\bD^\times/\mu_m)\ar[d]^-{\mathrm{inf}_{G/G_K}} \ar[r]\ar@{-->}[rd]^-\xi & H^2(G,\mu_m) \ar[d]^-{\mathrm{inf}_{G/G_K}} \\
H^1(G_K,\mu_m)\ar[r]& H^1(G_K,\bD^\times)\ar[r] & H^1(G_K,\bD^\times/\mu_m)\ar[r]^-\delta & H^2(G_K,\mu_m)
}
\end{equation}
where \(\mathrm{inf}_{G/G_K}:H^1(G,\bullet^{G_K})\rightarrow H^1(G_K, \bullet)\) denotes the inflation map. 

\begin{proposition}\label{prop}
 The inflation map identifies \(\ker(\phi)\cap\ker(\xi)\) with \(\Sha_m(G_K, \bD^\times)\).    
\end{proposition}
\begin{proof}
By Lemma \ref{lemma}, we have an inclusion \(\mathrm{inf}_{G/G_K}(\ker(\phi)\cap\ker(\xi))\subseteq\Sha_m(G_K, \bD^\times)\), 
so we need to show the opposite inclusion. Fix \(x\in \Sha_m(G_K, \bD^\times)\). There is a unique \(y\in \ker(\phi)\) such that \(x=\mathrm{inf}_{G/G_K}(y)\) and, since \(\delta(x)=1\), \(y\in \ker(\xi)\). 
\end{proof}

As a consequence of Proposition \ref{prop}, we see that
\((X, \varphi, m)\) satisfies the local-global principle if and only if \(\ker(\phi)\cap\ker(\xi)\) is trivial. 
In particular, we have the following corollaries. 

\begin{corollary}\label{coro1}
If \(H^1(G,\bD^\times/\mu_m) = 1\), then the local-global principle holds for \((X, \varphi, m)\).  
\end{corollary}
\begin{proof}
   If \(H^1(G,\bD^\times/\mu_m) = 1\), then \(\ker(\phi) = \ker(\xi) = 1\) and the result follows from Proposition \ref{prop}. 
\end{proof}

\begin{corollary}\label{coro2}
If \(\inf_{G/G_K} \colon H^2(G, \mu_m) \to H^2(G_K, \mu_m)\) is injective, then the inflation map identifies \(\Sha_m(G_K, \bD^\times)\) with
\[\ker(\phi)\cap \im(H^1(G, \bD^\times) \longrightarrow H^1(G, \bD^\times/\mu_m)).\]
In particular, if \(H^1(G, \bD^\times) = 1\) and any of the following holds:
\begin{enumerate}
    \item \(\mu_m^G\) is trivial, 
    \item \(H^2(G, \mu_m) = 1\),
\end{enumerate}
then the local-global principle holds for \(\rho\).   
\end{corollary}
\begin{proof}
Notice that \(\ker(\xi)\) is the pre-image of \(\ker(H^2(G,\mu_m)\xrightarrow{\inf_{G/G_K}}H^2(G_K, \mu_m))\) via the connecting homomorphism \(H^1(G, \bD^\times/\mu_m) \to H^2(G, \mu_m)\). The tail of the inflation-restriction exact sequence for \(\mu_m\) reads
\[
    \hom(G_L, \mu_m^G) \cong H^1(G_L, \mu_m)^G \longrightarrow H^2(G,\mu_m)\xrightarrow{\inf_{G/G_K}}H^1(G_K, \mu_m).
\]
In particular, under the injectivity hypothesis
\[
    \ker(\xi) = \ker(H^1(G, \bD^\times/\mu_m) \to H^2(G, \mu_m)) = \im(H^1(G, \bD^\times) \to H^1(G, \bD^\times/\mu_m)).
\]
Thus, the inflation map identifies \(\Sha_m(G_K, \bD^\times)\) with
\(\ker(\phi)\cap \im(H^1(G, \bD^\times) \to H^1(G, \bD^\times/\mu_m)),\)
which is trivial under the additional hypothesis that \(H^1(G, \bD^\times) = 1\).
\end{proof}

\subsection{Normal subgroups and coprime order} 
In the following we collect some formal results which ensure the validity of the local-global principle in the case where \(G\) admits a normal subgroup satisfying certain co-primality conditions.
Most of these results are not essential to the rest of the paper, but we state them here because they can be used in practice to compute Tate--Shafarevich sets.

\begin{proposition}\label{prop3}
Let \(\bD, D_K\) and \(G\) be as above and consider \(m\geq 2\) an integer. Assume that:\
\begin{enumerate}
    \item \(H^1(G,\bD^\times) = 1\),
    \item there is a normal subgroup \(N \lhd G\) whose order is coprime to \(m\) and
    \item \(H^2(G/N, \mu_m^N) = 1\).
    
\end{enumerate} Then, the local-global principle holds for \((X, \varphi, m)\).   
\end{proposition}
\begin{proof}
By Corollary \ref{coro2}, it is enough to show that \(H^2(G,\mu_m) = 1\).
By \cite[Corollary 1, p.\ 105]{CF}, we have that \(H^j(N, \mu_m) = 1,\) for \(j \ge 1\). This implies that the Hochschild--Serre spectral sequence collapses on the second page, since all the terms \(E^{ij}_2 = H^i(G/N, H^j(N, \mu_m)), j \neq 0,\) on the second page are nil.  By the convergence of the spectral sequence, we have 
\[
        H^2(G, \mu_m) = H^2(G/N, \mu_m^N) = 1. \qedhere 
\]\end{proof}

Let \(\chi \in H^1(G, \bD^{\times}/\mu_m)\) and pick \(c \colon G \to \bD^\times/\mu_m\) a \(1\)-cocycle representing \(\chi\). As in \cite[\S5.3]{Serre-Galois}, we define the twisted \(G\)-module \(\bD^{\times, \chi}\) to be the group \(\bD^{\times}\) with the \(G\)-action given by
\[
    \sigma \cdot_\chi a = c(\sigma)\sigma(a)c(\sigma)^{-1}, 
\]
where \(\sigma \in G, a \in \bD^{\times}\). Up to isomorphism of \(G\)-modules, \(\bD^{\times, \chi}\) only depends on \(\chi\), not on \(c\).

\begin{proposition}\label{prop: case of S3}
Let \(\bD, D_K\) and \(G\) be as above and consider \(m\geq 2\) an integer. Assume that:\
\begin{enumerate}
    \item \label{props3:H90G} \(H^1(G,\bD^{\times, \chi}) = 1\), for all \(\chi \in H^1(G, \bD^\times/\mu_m)\),
    \item \label{props3:coprim} there is a normal subgroup \(N \lhd G\) whose index in \(G\) is coprime to \(m\),
    \item \label{props3:dec} \(N\) is a decomposition group in \(G\) and
    \item \label{props3:H90N} \(H^1(N, \bD^{\times, \chi})^{G/N} = 1\), for all \(\chi \in H^1(G, \bD^\times/\mu_m)\).
\end{enumerate}
Then, the local-global principle holds for \((X, \varphi, m)\).   
\end{proposition}
\begin{proof}
Condition (\ref{props3:coprim}) implies that \(H^i(G/N, H^j(N, \mu_m)) = 1\), for \(i \neq 1\), since \(H^j(N, \mu_m)\) has order coprime to \((G:N)\). In particular, the Hochschild--Serre spectral sequence collapses on the second page and the restriction map \(\res \colon H^2(G, \mu_m) \to H^2(N, \mu_m)^{G/N}\) is an isomorphism. By restriction to \(N\), from the first row of \eqref{diaxi}, we obtain the commutative diagram
\[
\xymatrix{
    & H^1(G,\bD^\times) \ar[r] \ar[d]^{\res}
    & H^1(G,\bD^\times/\mu_m) \ar[r] \ar[d]^{\res}
    & H^2(G,\mu_m) \ar[d]^{\res} \\
    & H^1(N,\bD^\times)^{G/N} \ar[r]
    & H^1(N,\bD^\times/\mu_m)^{G/N} \ar[r]
    & H^2(N,\mu_m)^{G/N}.
}
\]
By Conditions \eqref{props3:H90G}, \eqref{props3:H90N} and the Corollary to \cite[Prop.~44]{Serre-Galois}, we deduce that the horizontal maps are injective. A simple diagram chase shows that \(\res \colon H^1(G,\bD^\times/\mu_m) \to H^1(N,\bD^\times/\mu_m)^{G/N}\) is injective. Since \(N = D_v\) for some \(v \in \Sigma_K\), by (\ref{props3:dec}), Proposition \ref{prop} entails that \(\Sha_m(G_K, \bD^\times) = 1\).
\end{proof}

\subsection{Trivial Galois action on \texorpdfstring{\(\mu_m\)}{roots of unity}} \label{ssec: muminvar}
Assume for this subsection that \(F\) is \(\ell\)-adic, \(\rho \colon G_K \to \GL_F(V)\) is a geometric Galois representation and \((X, \varphi)\) is the object associated with \(\rho\) in \S\ref{ssec: cat action on reps}. Let \(m \ge 1\).

\begin{theorem}\label{thm : coprimality}
If \(\mu_m \subseteq D_K\) and \(m\) is coprime to \(n\), then \((\rho, m)\) satisfies the local-global principle.
\end{theorem}
\begin{proof}
Let \(\chi \in \Sel_m(G_K, \bD^\times)\) and consider \(\rho^\chi\). The fact that \(\chi\) comes form \(\Sel_m(G_K, \bD^\times)\) means that for all \(v \in \Sigma^\infty_K\), there is some \(\chi_v \in \hom(G_K, \mu_m)\) such that
\[
    \rho^\chi_{|G_{K_v}} \sim \chi_v\rho_{|G_{K_v}}.
\]
Let \(e \in \Z_{\ge 1}\) be an inverse of \(n\) modulo \(m\). Consider \(\eta \coloneqq (\det(\rho^\chi)/\det(\rho))^e \colon G_K \to F^\times\). For all \(v \in \Sigma^\infty_K\), we have that \(\eta_{|D_v} = \chi_v\). By Chebotarev's density theorem, we deduce that \(\eta\) has image in \(\mu_m\). Moreover, by Brauer--Nesbitt, \(\rho^\chi \sim \eta\rho\). This shows that \(\chi \in \Twist_m(G_K, \bD^\times)\). In particular, \(\Sha_m(G_K, \bD^\times)\) is reduced to a point.
\end{proof}

\section{Finiteness of \texorpdfstring{\(\Sha_m\)}{the Tate--Shafarevich group}}
Let \((X, \varphi)\) be as in \S\ref{sec: def sha}. We show that \(\Sha_m(G_K, \bD^\times)\) is finite. The proof is based on an interpretation of \(H^1(G, \bD^{\times, \chi})\) in terms of the skew group ring \(\bD^\chi \rtimes G\), defined below. Lemma \ref{lemma: skew reps and H1} may be seen as rephrasing \cite[Proposition~I.33]{Serre-Galois}, by identifying \emph{principal homogeneous spaces} relative to the \(G\)-group \(\bD^{\times, \chi}\) with free \(\bD^\chi\)-modules of rank one with a ``compatible \(G\)-action''.

\begin{theorem}\label{theorem: finite sha}
The pointed set \(\Sha_m(G_K, \bD^\times)\) is finite.
\end{theorem}

Let \(\chi \in H^1(G, \bD^{\times}/\mu_m)\) and pick \(c \colon G \to \bD^\times/\mu_m\) a \(1\)-cocycle representing \(\chi\).
With the same formula we used for the definition of \(\bD^{\times, \chi}\), we can define a \(\chi\)-twisted \(G\)-module \(\bD^\chi\), so that \(\bD^{\times, \chi} = (\bD^\chi)^\times\). Moreover, both the standard action of \(G\) and the \(\chi\)-twisted one are given by automorphisms of \(\bD\) as an \(F\)-algebra. We begin by showing a key lemma. 

\begin{lemma}\label{lemma: finiteness of sha}
If \(H^1(G, \bD^{\times, \chi})\) is finite for all \(\chi \in H^1(G, \bD^{\times}/\mu_m)\), then \(\Sha_m(G_K, \bD^\times)\) is finite.
\end{lemma}
\begin{proof} By the exactness of the first row of \eqref{diaxi}, finiteness of \(H^2(G, \mu_m)\) and Corollary to Proposition 44 in \cite[\S5]{Serre-Galois}, the set \(H^1(G, \bD^\times/\mu_m)\) is finite. Proposition \ref{prop} lets us conclude.
\end{proof}
Thus, to prove Theorem \ref{theorem: finite sha}, it is enough to show the following.
\begin{proposition}\label{prop: cohom finiteness}
The pointed set \(H^1(G, \bD^{\times, \chi})\) is finite for all \(\chi \in H^1(G, \bD^{\times}/\mu_m)\).
\end{proposition}
In what follows, when we refer to the action of \(G\), we mean the one twisted by \(\chi\). The same arguments apply regardless of which \(\chi\) we pick, so we denote the action of \(G\) on \(\bD^\chi\) only by \(\sigma(a)\), for all \(\sigma \in G, a \in \bD^\chi\).

We consider the \emph{skew group ring} \(\bD^\chi \rtimes G\), which is a free left \(\bD^\chi\)-module with basis given by elements of \(G\) and product defined by
\[
    (a_\sigma \sigma)(a_\tau \tau) = (a_\sigma \sigma(a_\tau))(\sigma \tau),
\]
for all \(\sigma, \tau \in G, a_\sigma, a_\tau \in \bD^\chi\). The ring \(\bD^\chi \rtimes G\) is a finite-dimensional \(F\)-algebra, because \(\bD^\chi\) is. Moreover, \(\bD^\chi \rtimes G\) is semisimple as an \(F\)-algebra, because \(F\) has characteristic zero. 
We can think of \(\bD^\chi\) as a sub-\(F\)-algebra of \(\bD^\chi \rtimes G\).

For \(a \in \bD^\chi\), write \(\Lcal_a \colon \bD^\chi \to \bD^\chi, b \mapsto ab,\) and \(\Rcal_a\colon \bD^\chi \to \bD^\chi, b \mapsto ba\). In particular, \(\Lcal \colon \bD^\chi \to \End_F(\bD^\chi), a \mapsto \Lcal_a\), and \(\Rcal \colon \bD^\chi \to \End_F(\bD^\chi), a \mapsto \Rcal_a\), give two representations of the \(F\)-algebra \(\bD^\chi\) into itself and are called the \emph{left}, respectively \emph{right}, \emph{regular representation}. Consider
\[
    \Xcal \coloneqq\{f \in \hom_{\Alg_F}(\bD^\chi \rtimes G, \End_F(\bD^\chi))\mid f_{|\bD^\chi} = \Lcal, f(\sigma)(a) = \sigma(a)f(\sigma)(1_{\bD^\chi}), \forall a \in \bD^\chi, \sigma \in G\}.
\]
Here \(\hom_{\Alg_F}\) denotes homomorphisms of \(F\)-algebras. 
Thus, \(\Xcal\) is the set of lifts of the left regular representation of \(\bD^\chi\) to a \(\bD^\chi \rtimes G\)-module structure on \(\bD^\chi\) compatible with the action of \(G\) on \(\bD^\chi\).

Two module structures \(f, g \in \hom_{\Alg_F}(\bD^\chi \rtimes G, \End_F(\bD^\chi))\) are equivalent if and only if there is some \(\xi \in \Aut_F(\bD^\chi)\) such that \(f(x) \xi = \xi g(x)\), for all \(x \in \bD^\chi \rtimes G\), that is, \(f = \xi g \xi^{-1}\).
If the action of \(\xi \in \Aut_F(\bD^\chi)\) preserves \(\Xcal\), then it must satisfy \(\Lcal_a\xi = \xi \Lcal_a\) for all \(a \in \bD^\chi\), which forces \(\xi\) to be of the form \(\xi = \Rcal_{b^{-1}}\) for \(b \in \bD^{\times, \chi}\). 
On the other hand, the formula \[f \longmapsto  \!b \star f \coloneqq(x \mapsto (a \mapsto f(x)(ab)b^{-1} )),\] 
for all \(f \in \Xcal, x \in \bD^\chi \rtimes G, a \in \bD^\chi,  b \in \bD^{\times, \chi}\), defines a left action of \(\bD^{\times, \chi}\) on \(\Xcal\), which is the restriction of the conjugacy action of \(\Aut_F(\bD^\chi)\) on \(\hom_{\Alg_F}(\bD^\chi \rtimes G, \End_F(\bD^\chi))\) via \(\bD^{\times, \chi} \to \Aut_F(\bD^\chi), b \mapsto \Rcal_{b^{-1}}\). 

\begin{lemma}\label{lemma: finiteness of quotXcal}
The quotient \({}_{\bD^{\times, \chi} \backslash} \Xcal\) is finite.
\end{lemma}
\begin{proof}
By the discussion above, \({}_{\bD^{\times, \chi} \backslash} \Xcal\) is a set of isomorphism classes of representations of the \(F\)-algebra \(\bD^\chi \rtimes G\) into the \(F\)-vector space \(\bD^\chi\). 
Thus, \({}_{\bD^{\times, \chi} \backslash} \Xcal\) has cardinality less than or equal to that of the set \({}_{\Aut_F(\bD^\chi) \backslash}\hom_{\Alg_F}(\bD^\chi \rtimes G, \End_F(\bD^\chi))\)
, which is finite, because \(\bD^\chi \rtimes G\) is semisimple and finite-dimensional. In a diagram, we have
\[ \begin{tikzcd}
         \{\Aut_F(\bD^\chi)f \mid \xi f_{|\bD^\chi} \xi^{-1} \in \Xcal, \exists \xi \in \Aut_F(\bD^\chi)\} \arrow[d, hook] \arrow[r, "{{\sim}}"] & {}_{\bD^{\times, \chi} \backslash} \Xcal  \arrow[d, hook] 
         \\
         \res^{-1}(\Aut_F(\bD^\chi) \Lcal) \arrow[d, "{{\res}}"] \arrow[r, hook] & {}_{\Aut_F(\bD^\chi) \backslash}\hom_{\Alg_F}(\bD^\chi \rtimes G, \End_F(\bD^\chi)) \arrow[d, "{{\res}}"] \\
         \{\Aut_F(\bD^\chi) \Lcal\}\arrow[r, hook]& {}_{\Aut_F(\bD^\chi) \backslash}\hom_{\Alg_F}(\bD^\chi, \End_F(\bD^\chi)),
    \end{tikzcd}
\]
where \(\res\) is given by sending \(f \colon \bD^\chi \rtimes G \to \End_F(\bD^\chi)\) to \(f_{|\bD^\chi}\).    
\end{proof}

\begin{lemma}\label{lemma: skew reps and H1}
There is a natural bijection \(H^1(G, \bD^{\times, \chi}) \cong {}_{\bD^{\times, \chi} \backslash} \Xcal\).
\end{lemma}
\begin{proof}
We consider the map
\begin{align*}
    Z^1(G, \bD^{\times, \chi}) & \longrightarrow \Xcal,\\
    u & \longmapsto f_u \coloneqq \sum_{\sigma \in G} a_\sigma \sigma \longmapsto \left( a \mapsto \sum_{\sigma \in G} a_\sigma \sigma(a) u(\sigma)^{-1}\right).
\end{align*}
One can check that \(f_u\) is a morphism of rings using the fact that \(u\) is a \(1\)-cocycle. 
Moreover, if \(b \in \bD^{\times, \chi}\) and we replace \(u\) by \(\sigma \mapsto bu(\sigma)\sigma(b^{-1})\), we see immediately that \(f_u\) is replaced by \(b \star f_u\). Thus, \(u \mapsto f_u\) induces a natural map \(H^1(G, \bD^{\times, \chi}) \to {}_{\bD^{\times, \chi} \backslash} \Xcal\). Conversely, consider \(f \in \Xcal\) and set \(u_f(\sigma) = (f(\sigma)(1_{\bD^\chi}))^{-1}\). 
One can check that \(u_f\) is a \(1\)-cocycle. Moreover, if we replace \(f\) by \(b\star f\), \(b \in \bD^{\times, \chi}\), we get \(u_{b \star f}(\sigma) = bu_f(\sigma)\sigma(b^{-1})\).
Therefore, the cohomology class \([u_f]\) only depends on the equivalence class \([f]\).
Thus, \([u] \mapsto [f_u]\) and \([f] \mapsto [u_f]\) are well-defined and inverse to each other.  
\end{proof}
Lemma \ref{lemma: finiteness of quotXcal} and Lemma \ref{lemma: skew reps and H1} imply Proposition \ref{prop: cohom finiteness}. This concludes the proof of Theorem \ref{theorem: finite sha}.
\begin{remark}
In principle, if one has information about the simple modules of \(\bD^\chi \rtimes G\), for each \(\chi \in H^1(G_K, \bD^\times/\mu_m)\), the argument of Lemma \ref{lemma: finiteness of quotXcal} can be used to obtain explicit bounds on the cardinality of \(H^1(G, \bD^{\times, \chi})\) and, in turn, bounds on the cardinality of \(\Sha_m(G_K, \bD^\times)\). 
\end{remark}

\section{Application to modular forms}

We record here a consequence of our results, which should be compared with  
\cite[Theorem 1.1]{recraj}. Notice that, however, our techniques are different.

\begin{corollary}\label{corollary: lgp for mod forms}

Let \(f \in S_{k}(N, \epsilon)\) be a  normalised cuspidal newform of weight \(k \ge 2\), level \(N\), nebentypus \(\epsilon\) and Hecke field \(F_{f}\). For a prime \(\lambda\) of \(F_{f}\) dividing \(\ell\), take \(\rho \colon G_\Q \to \GL_2(F_{f, \lambda})\) the associated Galois representation.
\begin{enumerate}[(i)]
    \item If \(f\) has no CM, then \(G = 1\) and \((\rho, m)\) satisfies the local-global principle for all \(m \ge 1\).
    \item If \(f\) has CM, then \(G \cong C_2\) and \((\rho, m)\) satisfies the local-global principle for all \(m \ge 1\).
\end{enumerate}
\end{corollary}
\begin{proof}
The form \(f\) has CM by a quadratic imaginary extension \(M\) of \(\Q\) if and only if \(\rho = \Ind_{G_M}^{G_\Q} \psi\), for some some Hecke character \(\psi\) of \(G_M\). From \cite[Proposition~4.2, Proposition~4.4]{ribetgalrepseigen}, we deduce that \(f\) has no CM if and only if \(\rho\) is strongly absolutely irreducible (if it were reducible restricting to \(H \subseteq G_\Q\) open and extending scalars from \(F_{f, \lambda}\) to some extension \(E\), then \(\rho(H)\) would be abelian, already over \(F_{f, \lambda}\)).
Thus, by Lemma \ref{lemma: strongly abs irred min end field}, for \(f\) without CM, we have \(D_L = D_\Q\) and the conclusion follows from Corollary \ref{cor : inertiffcyc}.  
For \(f\) with CM, \(\rho_{|G_M} \cong \psi \oplus \psi^c, c \in \Gal(M/\Q), c \neq 1,\) and, since \(\psi\neq \psi^c\), \(D_M \cong F_{f, \lambda} \times F_{f, \lambda}\). Then \(L = M\). This concludes the proof, by Corollary \ref{cor : inertiffcyc}.
\end{proof}

\section{Applications to abelian varieties}

Let \(A\) be an abelian variety over the number field \(K\). Write \( L=L_\mr{end}(A)\) for the minimal endomorphism field of \(A\). Recall that this is the same as the minimal endomorphism field of \(\rho_{A, \ell}\), by Lemma \ref{lemma: min endo of ab is repth}. For \(M\) any finite extension of \(K\), let us write \(D_M(A)\), or simply \(D_M\) when \(A\) is clear from context, instead of \(\End^0(A_M)\).

Let \(m \ge 1\) be an integer. Like in Remark \ref{sssec: extra end}, we fix \(E\) any commutative subfield of \(D_K(A)\) central in \(\bD(A)\). Notice that \(E\) is a number field. In this section, we make the following assumption.

\begin{assumption}\label{ass: mum in end for ab vars}
We assume that \(\bD(A)\) admits a structure of \(E(\zeta_m)\)-algebra and we fix one. Moreover, we assume that \(\mu_m \subseteq \bD(A)^\times\) is stable under the action of \(G_K\).
\end{assumption}

Fix \(\lambda\) a prime of \(E\) above \(\ell\) and write \(E_\lambda\) for the completion of \(E\) at \(\lambda\). We consider the natural representation \(\rho_{A, \lambda} \colon G_K \to \GL_{E_\lambda}(V_\lambda(A))\). Assumption \ref{ass: mum in end for ab vars} for \(A\) implies Assumption \ref{ass1} for \(\rho_{A, \lambda}\).
\begin{lemma}\label{lemma: if lgp for rep then lgp for abvar}
For all \(\lambda\), there is a natural injection \[\Sha_m(G_K, \bD(A)^\times) \hookrightarrow \Sha_m(G_K, \bD(\rho_{A, \lambda})^\times).\]
In particular, if the local-global principle holds for \((\rho_{A, \lambda}, m)\), for some $\lambda$, then it holds for \((A, m)\). 
\end{lemma}
\begin{proof}
From the inclusions \(\mu_m \subseteq \bD(A) \subseteq \bD(A) \otimes_E E_\lambda\) and Faltings's theorem, we see that there is a well defined morphism
\[\begin{tikzcd}
	1 & {\Twist_m(G_K, \bD(A)^\times)} & {\Sel_m(G_K, \bD(A)^\times)} & {\Sha_m(G_K, \bD(A)^\times)} & 1 \\
	1 & {\Twist_m(G_K, \bD(\rho_{A, \lambda})^\times)} & {\Sel_m(G_K, \bD(\rho_{A, \lambda})^\times)} & {\Sha_m(G_K, \bD(\rho_{A, \lambda})^\times)} & 1
	\arrow[from=1-1, to=1-2]
	\arrow[from=1-2, to=1-3]
	\arrow[from=1-3, to=1-4]
	\arrow[from=1-4, to=1-5]
	\arrow[from=2-1, to=2-2]
	\arrow[from=2-2, to=2-3]
	\arrow[from=2-3, to=2-4]
	\arrow[from=2-4, to=2-5]
    \arrow[from=1-2, to=2-2]
    \arrow[from=1-3, to=2-3]
    \arrow[from=1-4, to=2-4]
\end{tikzcd}\]
between the short exact sequences associated with \(A\) and \(\rho_{A, \lambda}\) in (\ref{FES}). Moreover, by definition, the map \(\Twist_m(G_K, \bD(A)^\times) \to \Twist_m(G_K, \bD(\rho_{A, \lambda})^\times)\) is surjective.

We claim that the natural map \(H^1(G_K, \bD(A)^\times) \to H^1(G_K, \bD(\rho_{A, \lambda})^\times)\) is injective. Consider \(\chi_1, \chi_2 \in H^1(G_K, \bD(A)^\times)\) and denote by \(A^{\chi_1}\) and \(A^{\chi_2}\) some choices of twists of the abelian variety \(A\) in the equivalence classes corresponding to \(\chi_1\) and \(\chi_2\), respectively. The classes \(\chi_1\) and \(\chi_2\) have the same image in \(H^1(G_K, \bD(\rho_{A, \lambda})^\times)\) if and only if there is an isomorphism of representations \(\xi_\lambda \colon \rho_{A^{\chi_1}, \lambda} \xrightarrow{\sim} \rho_{A^{\chi_2}, \lambda}\). Via the isomorphism
\(
    \hom^0(A^{\chi_1}, A^{\chi_2}) \otimes_{E} E_\lambda \xrightarrow{\sim} \hom_{E_\lambda[G_K]}(\rho_{A^{\chi_1}, \lambda}, \rho_{A^{\chi_2}, \lambda}),
\)
we can approximate \(\xi_\lambda\) by an element of \(\hom^0(A^{\chi_1}, A^{\chi_2})\), which will be invertible if it is sufficiently close to \(\xi_\lambda\) in the \(\lambda\)-adic norm. This shows that \(A^{\chi_1}\) is isogenous to \(A^{\chi_2}\) and thus, as twists, they are equivalent. This means that \(\chi_1 = \chi_2\), by Proposition \ref{proposition: twists and cohomology}, and we conclude that \(H^1(G_K, \bD(A)^\times) \to H^1(G_K, \bD(\rho_{A, \lambda})^\times)\) is injective. This implies that \(\Twist_m(G_K, \bD(A)^\times) \to \Twist_m(G_K, \bD(\rho_{A, \lambda})^\times)\) is a bijection, and that the central arrow \(\Sel_m(G_K, \bD(A)^\times) \to \Sel_m(G_K, \bD(\rho_{A, \lambda})^\times)\) in the diagram above is injective.

By \cite[Corollary to Proposition I.44, \S5]{Serre-Galois}, \(S(\bullet) \coloneqq \ker(H^1(G_K, \bD(\bullet)^\times/\mu_m) \to H^2(G_K, \mu_m)), \bullet = A, \rho_{A, \lambda},\) is naturally identified with the quotient of \(H^1(G_K, \bD^\times)\) by the action of \(H^1(G_K, \mu_m)\). In particular, the natural map \(S(A) \to S(\rho_{A, \lambda})\) can be identified with
\[
    \frac{H^1(G_K, \bD(A)^\times)}{H^1(G_K, \mu_m)} \longrightarrow \frac{H^1(G_K, \bD(\rho_{A, \lambda})^\times)}{H^1(G_K, \mu_m)},
\]
which is also injective. Since \(\Sha_m(G_K, \bD(\bullet)^\times) \subseteq S(\bullet)\), we have the desired injectivity.

In particular, when \(\Sha_m(G_K, \bD(\rho_{A, \lambda})^\times)\) is reduced to a point, so is \(\Sha_m(G_K, \bD(A)^\times)\)
and the local-global principle holds for \((A, m)\).
\end{proof}

\subsection{The geometrically simple abelian case} \label{sec: simp CM}

\label{sec: commutative simple} Let us write \(D_M\) instead of \(D_M(A)\). In this subsection, we work under the following hypothesis (which, by Albert's classification, \cite[Theorem 2, Chapter IV.21]{av}, is equivalent, when \(m\geq 3\), to \(\bD\) being a CM field). 

\begin{assumption}\label{ass2}
    Suppose that \(A\) is geometrically simple and \(\bD\) is commutative. In particular, \(D_K = \bD^G\) is also commutative and we can take \(E = D_K\) in Assumption \ref{ass: mum in end for ab vars}, which we also assume.
\end{assumption}
By minimality of \(L\), the natural morphism \(G \longrightarrow \Gal(\bD/D_K)\) is injective, hence it is an isomorphism. We have proved the following.

\begin{lemma}\label{lemmaCM} If \(\bD\) is commutative, then the canonical action of \(G_K\) on \(\bD\) induces an isomorphism of groups \(G \simeq \Gal(\bD/D_K)\). 
\end{lemma}

It follows from Lemma \ref{lemmaCM} that, for all \(i\geq 0\), we have canonical isomorphisms of groups 
\begin{equation}\label{galois}
H^i(G,\bD^\times)\simeq H^i(\Gal(\bD/D_K),\bD^\times).\end{equation}

\begin{corollary}\label{coroarith} Let \(\sharp G\) be the cardinality of \(G\). If \((m,\sharp G)=1\), then the local-global principle holds for \((A,m)\).      
\end{corollary}

\begin{proof} Since \((m, \sharp G)=1\), we have \(H^2(G,\mu_m) = 1\) (by \cite[Corollary 1, p.\ 105]{CF}). 
Also, \(H^1(G,\bD^\times) = 1\) by \eqref{galois} and Hilbert 90. The result follows then from Corollary \ref{coro2}. 
\end{proof}

\begin{theorem}\label{thm: main_Gal}
For all (odd) \(m \ge 3\) such that \(\mu_m^G = \{1\}\), \((A,m)\) satisfies the local-global principle.    
\end{theorem}

\begin{proof}
By Hilbert 90 and Lemma \ref{lemmaCM}, \(H^1(G,\bD^\times)=1\). Thus, one can apply Corollary \ref{coro2}.
\end{proof}

\begin{example}\label{example: ec}
Let \(E/\Q\) be an elliptic curve with CM by a (quadratic imaginary) field \(L\). For all \(m\) odd such that \(\mu_m \subseteq L\), \((E, m)\) satisfies the local-global principle. Indeed, this follows immediately from Theorem \ref{thm: main_Gal}. Notice that since \(\mu_m \subseteq L\), the only possible choice is \(m = 3\).
\end{example}

\begin{example}
\label{example: GGL}
Let \(J_d/\Q\) be the Jacobian variety of the projective curve \(C_d/\Q\) given by the affine equation \(y^2=x^d+1\). 
Since the map \((x,y)\mapsto (\zeta_dx,y)\) is an automorphism of \(C_d/\overline{\Q}\), where \(\zeta_d\) is a primitive \(n\)-th root of unity, we see that \(\Q(\mu_d)\hookrightarrow \bD\). Moreover, \(D = \bD^G = \Q\) in this case.
By \cite[Theorem 3.0.1]{GGL}, the abelian variety \(J_d\) admits a canonical factorization up to isogeny over \(\Q\) 
which contains exactly one factor \(A_{d,m}\) for each positive divisor \(m\) of \(d\), except \(m=1\) and \(m=2\); moreover, if \(m\) is odd, then \(A_{d, m}\) is geometrically simple and has complex multiplication by \(\Q(\zeta_m)\). 
Then, Theorem \ref{thm: main_Gal} implies that, for any \(d\) and \(m\geq 3\) an odd divisor of \(d\), \((A_{d, m}, m)\) satisfies the local-global principle.
\end{example}

\begin{remark}
The content of the Examples \ref{example: ec} and \ref{example: GGL} also follows from Proposition \ref{prop3}.
\end{remark}

\begin{remark}
In the setup of Example \ref{example: GGL}, taking \(d = 3p\) for \(p \equiv 13 \pmod{24}\) a prime, \cite{ACF} constructs a counter-example for quadratic twists of abelian varieties. This counter-example relies heavily on the control that one has on the existence of subextensions of \(\Q(\zeta_d)/\Q\) of degree a power of \(2\), the only prime for which a lot of what we prove here breaks down. 
\end{remark}

\subsection{Further applications} \label{ssec: further apps}

By \cite[Theorem 2, Chapter IV.21]{av} and Assumption \ref{ass: mum in end for ab vars}, for \(A\) geometrically simple of dimension \(g\), Euler's totient \(\phi(m)\) divides the degree of the endomorphism algebra \(\bD\) over \(\Q\), which in turn is a divisor of \(2g\). For \(m \geq 3\) odd, this implies that:
 \begin{itemize}
        \item If \(g=2^a\) for any \(a \ge 0\), then \(m\) is a squarefree product of Fermat primes.
        \item If \(g=3\), then \(m=3\) or \(m=7\) or \(m=9\).
        \item If \(g=5\), then \(m=3\) or \(m=11\).
        \item If \(g=6\), then \(m \in \{3,5,7,9,13,21\}\).
        \item If \(g=7\), then \(m=3\).
\end{itemize}

\begin{corollary}\label{cor : easylgp}
    Let \(A\) be as above and assume that \(G\) acts trivially on \(\mu_m\), see \S\ref{ssec: muminvar}.
    If \(g=2^a\), for some \(a \ge 0\), or \(g \leq 7\), then, for all odd \(m \ge 3\), \((A,m)\) satisfies the local-global principle. 
\end{corollary}

\begin{proof}
To apply Theorem \ref{thm : coprimality} and Lemma \ref{lemma: if lgp for rep then lgp for abvar}, we can assume that \(\mu_m \subseteq E\) in Assumption \ref{ass: mum in end for ab vars}. In particular, \(n = \dim_F(V_\ell(A))\) divides \({2g}/\phi(m)\), since \(\phi(m) = [\Q(\mu_m) \colon \Q]\). Thus, whenever \(\phi(m)  = 2g\), then \(n = 1\) and the local-global principle holds. In all the following cases Theorem \ref{thm : coprimality} applies:
\begin{itemize}
    \item If \(g=2^a\), then \(n\) is also a power of \(2\), thus coprime to \(m\).
    \item If \(g=5\), then \(m=3\) or \(11\). Since \(n\) is a divisor of \(10\), in both cases we have \((m,n)=1\).
    \item If \(g=7\), then \(m=3\), which is coprime to \(2g=14\), hence to \(n\).
    \item If \(g=3\) or \(6\) and \(m \ge 5\), a direct computation in the cases listed above shows that \((m, n)=1\).
\end{itemize}
In the remaining cases, \(g = 3\) or \(6\) and \(m = 3\). Let \(\delta^2\) denote the dimension of \(\bD\) over its center \(Z(\bD)\). Also write \(e_0\) for the degree of the maximal totally real subextension of \(Z(\bD)\) over \(\Q\).
Our assumptions imply that \(\bD\) is of Type IV, and, by Albert's classification, we have \(e_0 \delta^2 \mid g\). Since \(3\) and \(6\) are squarefree, \(\bD\) is a commutative CM number field. Therefore
\[
\Q(\zeta_3) \subseteq D_K \subseteq \bD
\]
is a tower of field extensions, and \([\bD : \Q] \mid 2g\). Let \(G=\Gal(\bD/D_K)\).

Assume \(g=3\). Then \(G\) is either trivial or cyclic of order \(3\). In the first case, we have \(\bD=D_K\) and the local-global principle holds by Corollary \ref{cor : inertiffcyc}. In the second case, it holds by Corollary \ref{cor : inertiffcyc}.

Assume \(g=6\). Then \([D_L : D_K] \mid 6\), so \(G\) is either cyclic or isomorphic to \(S_3\). Again, the case when \(G\) is cyclic follows from Corollary \ref{cor : inertiffcyc}. The case \(G \cong S_3\) is covered by Proposition \ref{prop: case of S3} with \(N\) the unique \(3\)-Sylow of \(G\). Since \(N\) is cyclic, by Chebotarev's density theorem, it is the decomposition group of some unramified prime. Moreover, \(H^1(G, \bD^\times) = 1 = H^1(N, \bD^\times)\), by Hilbert 90.
\end{proof}


\bibliographystyle{alpha}
\bibliography{bib.bib}

\end{document}